\documentclass[a4paper,12pt]{article}
\usepackage{setspace}
\usepackage{latexsym} 
\usepackage[utf8]{inputenc}
\usepackage[english]{babel}
\usepackage{float}
\usepackage{amsfonts,amsmath,amssymb,amsbsy,amsthm,enumerate}
\usepackage[mathscr]{euscript}
\usepackage{pstricks}
\usepackage{multirow}
\usepackage{verbatim}
\usepackage{color}
\usepackage{enumerate}
\usepackage[normalem]{ulem} 
\usepackage{hyperref}

\newtheorem{theorem}{Theorem}
\newtheorem{problem}[theorem]{Problem}
\newtheorem{corollary}[theorem]{Corollary}
\newtheorem{conjecture}[theorem]{Conjecture}
\newtheorem{lemma}[theorem]{Lemma}
\newtheorem{claim}{Claim}[theorem]

\usepackage{lineno}

\newcommand{\ifun}{{\rm i}}
\newcommand{\andrasfai}{Andr\'asfai}

\newcommand*\oline[1]{%
  \,\vbox{%
    \hrule height 0.5pt
    \kern0.25ex
    \hbox{%
      \kern-0.1em
      \ifmmode#1\else\ensuremath{#1}\fi
      \kern0.1em
    }
  }
}

\newcommand{\Gcompl}{\oline{G}}
\newcommand{\ceil}[1]{\lceil #1 \rceil}

\usepackage{tikz}
\usepackage[subrefformat=parens,labelformat=parens]{subcaption}
\usetikzlibrary{calc,positioning,decorations.pathmorphing,decorations.pathreplacing}
\tikzset{black vertex/.style={circle,draw,minimum size=1mm,inner sep=0pt,outer sep=2pt,fill=black, color=black}}

\title{Immersions of large cliques in graphs \\with independence number~2 \\and bounded maximum degree}

\usepackage{authblk}
\author[1]{Fábio~Botler}
\author[1]{Cristina~G.~Fernandes}
\author[2]{Carla~N.~Lintzmayer}
\author[3]{Rui~A.~Lopes}
\author[4]{Suchismita~Mishra}
\author[3]{Bruno~L.~Netto}
\author[2]{Maycon~Sambinelli}
\affil[1]{\small Universidade de São Paulo, São Paulo, Brazil}
\affil[2]{\small Universidade Federal do ABC, Santo André, Brazil}
\affil[3]{\small Universidade Federal do Rio de Janeiro, Rio de Janeiro, Brazil}
\affil[4]{\small Universidad Andrés Bello, Santiago, Chile}

\date{}

\begin{document}

\maketitle

\begin{abstract}
  An immersion of a graph~\(H\) in a graph~\(G\) is a minimal subgraph~\(I\)
  of~\(G\) for which there is an injection \(\ifun \colon V(H) \to V(I)\) and
  a set of edge-disjoint paths \(\{P_e: e \in E(H)\}\) in~\(I\) such that
  the end vertices of~\(P_{uv}\) are precisely~\(\ifun(u)\) and~\(\ifun(v)\).
  The immersion analogue of Hadwiger Conjecture~(1943)\nocite{hadwiger1943klassifikation},
  posed by Lescure and Meyniel~(1985)\nocite{lescure41problem}, asks whether
  every graph~\(G\) contains an immersion of~\(K_{\chi(G)}\).
  Its restriction to graphs with independence number~2 has received some
  attention recently, and Vergara~(2017)\nocite{vergara2017complete} raised the weaker conjecture
  that every graph with independence number~2 has an immersion of~\(K_{\chi(G)}\).
  This implies that every graph with independence number~2 has an
  immersion of~\(K_{\ceil{n/2}}\).
  In this paper, we verify Vergara Conjecture for graphs with bounded 
  maximum degree.
  Specifically, we prove that if~\(G\) is a graph with independence 
  number~\(2\), maximum degree less than \(2n/3 - 1\) and clique covering number at most \(3\),
  then~\(G\) contains an immersion of~\(K_{\chi(G)}\) (and thus of~\(K_{\ceil{n/2}}\)).
  Using a result of Jin (1995), this implies that if~\(G\) is a graph with independence 
  number~\(2\) and maximum degree less than \(19n/29 - 1\),
  then~\(G\) contains an immersion of~\(K_{\chi(G)}\) (and thus of~\(K_{\ceil{n/2}}\)).
\end{abstract}

\section{Introduction}

In this paper, every graph is simple, that is, contains no loops or multiple edges.
Given a graph~\(G\), we denote by~\(\chi(G)\), \(\alpha(G)\), \(\delta(G)\)
and~\(\Delta(G)\) the usual chromatic number, independence number, and minimum
and maximum degree of \(G\), respectively.
We consider the problem of finding special subgraphs in dense graphs.
Specifically, we are interested in a problem related to the following 
conjecture posed by Hadwiger~\cite{hadwiger1943klassifikation}.

\begin{conjecture}[Hadwiger, 1943]
\label{conj:hadwiger}
  Every graph \(G\) contains \(K_{\chi(G)}\) as a minor.
\end{conjecture}

Conjecture~\ref{conj:hadwiger} is still open, but it has been verified
in many cases, as for graphs with chromatic number at most~6~\cite{robertson1993hadwiger}.
An approach that turned out to be fruitful in exploring Conjecture~\ref{conj:hadwiger}
is to parametrize it by the independence number.
Indeed, a graph on~\(n\) vertices with independence number~\(\alpha\) has
chromatic number at least \(\lceil n/\alpha\rceil\), and Conjecture~\ref{conj:hadwiger}
would imply that such a graph has a minor of \(K_{\ceil{n/\alpha}}\).
In this direction, Duchet and Meyniel~\cite{duchet1982hadwiger} proved that
a graph on~\(n\) vertices with independence number~\(\alpha\) has a minor
of \(K_{\ceil{n/(2\alpha - 1)}}\); and, after several partial results 
(see~\cite{fox2010complete}), Balogh and Kostochka~\cite{balogh2011large} 
further improved this result by proving that every such graph has a minor
of \(K_{\ceil{n/((2-c)\alpha)}}\), where \(c \approx 0.0521\).

As pointed out by Quiroz~\cite{quirozpersonal}, special attention has been given
to the case \(\alpha = 2\), and important results explore graphs with
small clique covering number or small fractional clique covering number.
For example, 
suppose that \(G\) is a graph with an even number \(n\) of vertices for which \(\alpha(G) \leq 2\).
Blasiak~\cite[Theorem 1.3]{blasiak2007special} proved that \(G\) contains a minor of \(K_{n/2}\) if~\(G\) has fractional clique covering number less than \(3\) or if~\(G\) has clique covering number~\(3\);
and Chudnovsky and Seymour~\cite[Result 1.3]{chudnovsky2012packing}
proved that~\(G\) contains a minor of \(K_{n/2}\) if~\(G\) contains a clique of size at least \(n/4\)
and, consequently, \(G\) contains a minor of \(K_{n/2}\) if~\(G\) has clique covering number at most \(4\).

In this paper, we are interested in the following immersion analogue of
Conjecture~\ref{conj:hadwiger}, posed by Lescure and Meyniel~\cite[Problem 2]{lescure41problem}.
An \emph{immersion} of a graph~\(H\) in a graph~\(G\) is a minimal subgraph~\(I\) 
of~\(G\) for which there is an injection \(\ifun \colon V(H) \to V(I)\) and 
a set of edge-disjoint paths \(\{P_e : e \in E(H)\}\) in~\(I\) such that 
the end vertices of~\(P_{uv}\) are precisely~\(\ifun(u)\) and~\(\ifun(v)\).
The vertices of~\(G\) in the image of \(\ifun\) are called the \emph{branch vertices} 
of the immersion. 
Moreover, we say that such an immersion is \emph{strong} if the internal
vertices of the paths~\(P_e\) are not branch vertices.

\begin{problem}[Lescure--Meyniel, 1985]
\label{problem:LescureMeyniel}
  Does every graph \(G\) contain a strong immersion of~\(K_{\chi(G)}\)?
\end{problem}

The weaker version of Problem~\ref{problem:LescureMeyniel} for (not strong) 
immersions was posed as a conjecture by Abu-Khzam and Langston~\cite{abu2003graph}.
Similarly to Conjecture~\ref{conj:hadwiger}, Problem~\ref{problem:LescureMeyniel}
and its weakening received special attention in the case of
graphs with independence number~2.
In particular, in 2017 Vergara~\cite[Conjecture 2]{vergara2017complete} posed the following 
restriction. 

\begin{conjecture}[Vergara, 2017]
\label{conj:vergaraChi}
  Every graph \(G\) with independence number~$2$
  contains an immersion of \(K_{\chi(G)}\).
\end{conjecture}

Observe that if a graph~\(G\) on~\(n\) vertices with independence
number~2 contains an immersion of \(K_{\chi(G)}\), then it contains 
an immersion of \(K_{\ceil{n/2}}\).
In fact, Vergara~\cite[Theorem~1.5]{vergara2017complete} proved that Conjecture~\ref{conj:vergaraChi}
is equivalent to the following (see~\cite[Conjecture~3]{vergara2017complete}).

\begin{conjecture}[Vergara, 2017]
\label{conj:vergaraN}
  Every graph on~\(n\) vertices with independence number~$2$
  contains an immersion of \(K_{\ceil{n/2}}\).
\end{conjecture}

In this direction, Vergara~\cite[Theorem~1.6]{vergara2017complete} proved that every graph on~\(n\) vertices
with independence number~2 has a strong immersion of \(K_{\ceil{n/3}}\).
This result was improved by Gauthier, Le, and Wollan~\cite[Theorem 1.7]{gauthier2019forcing},
who proved that every such graph contains a strong immersion of \(K_{2\lfloor n/5\rfloor}\).

\begin{theorem}[Gauthier--Le--Wollan, 2019]
\label{thm:gauthier}
  Every graph on~\(n\) vertices with independence number~$2$
  contains a strong immersion of \(K_{2\lfloor n/5\rfloor}\).
\end{theorem}

In 2021, Quiroz~\cite{quiroz2021clique} verified 
Conjecture~\ref{conj:vergaraN} for graphs with independence number \(2\) and special forbidden subgraphs.
In 2024, Botler, Jiménez, Lintzmayer, Pastine, Quiroz, and 
Sambinelli~\cite[Theorem 3]{botler2024biclique} proved that every such graph
contains an immersion of every complete bipartite graph with~\(\ceil{n/2}\)
vertices (see also~\cite{chen4882255simple}).
The main strategy of Botler et al.\ is to partition the vertex set 
of the given graph into five parts, say \(V_1,\ldots, V_5\), and to show that
they are somehow ``cyclically connected'', that is, the vertices of \(V_i\) 
are adjacent to many vertices of \(V_{i+1}\) for \(i=1,\ldots, 5\), where $V_6=V_1$. 
A special case of this partition is when \(G\) contains a spanning
complete-blow-up of~\(C_5\) (the graph obtained from~\(C_5\) by replacing 
each vertex with a clique and by replacing each edge with a complete
bipartite graph).
In this case, Quiroz~\cite[Lemma 2.2]{quiroz2021clique} also proved 
the existence of an immersion of~\(K_{\chi(G)}\), and hence of~\(K_{\ceil{n/2}}\).

One can expect that vertices of high degree help when looking for
large clique immersions.
For example, Vergara's proof of the 
existence of \(K_{\ceil{n/3}}\) immersions in graphs with 
independence number~2 starts by observing that a counterexample 
must have minimum degree at least \(\lfloor 2n/3\rfloor\) and, analogously, 
a step of the proof of Theorem~\ref{thm:gauthier} is to prove that
in a counterexample every vertex contained in an induced \(C_5\) has degree at least \(3n/5\).
Similarly, it is not hard to prove that any minimum counterexample
to Conjecture~\ref{conj:vergaraN} has no pair of nonadjacent vertices
with at least \(\ceil{n/2} - 2\) common neighbors.
In this paper, we consider graphs without vertices of large degree.
Specifically, we answer Problem~\ref{problem:LescureMeyniel} positively for graphs with independence number 2 and maximum degree bounded as~follows.

\begin{theorem}
\label{thm:mainChi}
  Let~\(G\) be a graph on~\(n\geq 11\) vertices with independence number~$2$. If~\({\Delta(G) < 19n/29 - 1}\), then \(G\) contains a strong immersion of \(K_{\chi(G)}\).
\end{theorem}

In fact, Theorem~\ref{thm:mainChi} is a consequence of the following result
and a result of Jin~\cite{jin1995triangle} stated as Theorem~\ref{thm:jin} ahead.
\begin{theorem}\label{thm:mainCliqueCovering}
  Let~\(G\) be a graph on~\(n\) vertices with independence number~\(2\)
  and clique covering number at most \(3\).
  If \(\Delta(G) < 2n/3 - 1\), 
  then \(G\) contains a strong immersion of~\(K_{\chi(G)}\).    
\end{theorem}

Naturally, Theorems~\ref{thm:mainChi} and~\ref{thm:mainCliqueCovering} imply the following.

\begin{corollary}
\label{cor:mainN}
  Let \(G\) be a graph on~\(n\) vertices with independence number~\(2\).
  If any of the conditions below  holds, then \(G\) contains a strong immersion of \(K_{\ceil{n/2}}\).
  \begin{itemize}
      \item[(i)] \(\Delta(G) < 19n/29 - 1\); or
      \item[(ii)] \(\Delta(G) < 2n/3 - 1\) and \(G\) has clique covering number at most \(3\).
  \end{itemize}
\end{corollary}

Our proof explores properties of the complement of the studied graph.
Specifically, we use the fact that triangle-free graphs with 
high minimum degree are homomorphic to the well-known 
\andrasfai~graphs (see Section~\ref{sec:andrasfai}).
Since \(C_5\) is an \andrasfai~graph, our result generalizes the case \(\alpha(G) = 2\) 
of the aforementioned result of Quiroz for graphs containing
complements of blow-ups of \andrasfai~graphs.
Indeed, our result is a consequence of a slightly more 
general result for graphs with a special \(3\)-clique cover (see Theorem~\ref{thm:graphs_with_special_clique_coloring}) and 
the fact that \andrasfai~graphs admit a corresponding proper coloring.

In fact, we prove the stronger statement that \(V(G)\) can be
partitioned into two sets \(A\), \(B\) such that 
(i) \(A\) induces a clique in \(G\), and 
(ii) \(G\) contains an immersion of a clique whose set of branch vertices 
is precisely~\(B\).
One of the main ideas of our proof is to use Hall's Theorem 
to identify, for each vertex \(u\in B\), a vertex \(r_u \in A\) 
that ``represents'' \(u\) in the sense that, when \(u\) has 
a missing adjacency, say \(uv \notin E(G)\) with \(v\in B\), 
we ``fix'' it by finding a path from~\(u\) to~\(v\) through~\(r_v\) and~\(r_u\).
The rest of the proof is to show that such paths are edge-disjoint.

Finally, although not making it explicit, 
the immersions found in this paper are \emph{totally odd} immersions,
meaning that each path in the immersion has an odd number of edges (see, for instance,~\cite{EcheverriaJimenezMishraQuirozYepez25}).

\paragraph{Organization of the paper.}
In Section~\ref{sec:andrasfai}, we present the \andrasfai~graphs,
which play an important role in this paper, and describe a special
coloring for them.
In Section~\ref{sec:main-result}, we use this coloring to
prove Theorem~\ref{thm:mainChi}.
Additionally, in Section~\ref{sec:2/5}, we present a somewhat simpler
proof of Theorem~\ref{thm:gauthier} that we believe to 
properly reveal the approach of Gauthier, Le, and Wollan~\cite{gauthier2019forcing}.
In particular, this presentation exposes a loose inequality 
that supports that Theorem~\ref{thm:gauthier} may not be tight.
This might be explored in further works.

\section{\andrasfai~graphs}
\label{sec:andrasfai}

Given graphs \(G\) and~\(H\), a \emph{homomorphism} from \(G\) to~\(H\)
is a function \(h\colon V(G) \to V(H)\) such that \(h(u)h(v) \in E(H)\) for every \(uv\in E(G)\).
When such a function exists, we say that \(G\) is \emph{homomorphic} to \(H\).

Let \(G\) be a triangle-free graph with \(n\) vertices.
\andrasfai~\cite{andrasfal1964graphentheoretische} showed that 
if \({\delta(G) > 2n/5}\), then \(G\) is bipartite.
This result was generalized in many directions, one of which is the following.
H{\"a}ggkvist~\cite{haggkvist1982odd} proved that if \(\delta(G) > 3n/8\), 
then \(G\) is \(3\)-colorable, and Jin~\cite[Theorem 9]{jin1995triangle} weakened 
this minimum degree condition proving that if \(\delta(G) > 10n/29\), 
then~\(G\) is \(3\)-colorable.
\begin{theorem}[Jin, 1995]\label{thm:jin}
    Let \(G\) be a triangle-free graph with \(n\geq 11\) vertices.
    If~\({\delta(G) > 10n/29}\), then \(\chi(G) \leq 3\).
\end{theorem}
Chen, Jin, and Koh~\cite[Theorem 3.8]{chen1997triangle} strengthened Theorem~\ref{thm:jin}
by exposing the structure of triangle-free graphs with chromatic number \(3\) 
and high minimum degree.
Specifically, they proved the following result, where \(\Gamma_d\) is
the graph \((V_d,E_d)\) for which \(V_d = [3d-1]\) and
\({E_d = \left\{xy : y = x + i \text{ with } i \in [d,2d-1]\right\}}\),
with arithmetic modulo~\(3d-1\).
The graphs \(\Gamma_d\) for \(d \in \mathbb{N}\) are known as 
the \emph{\andrasfai~graphs} (see Figure~\ref{fig:andrasfai_graph}).

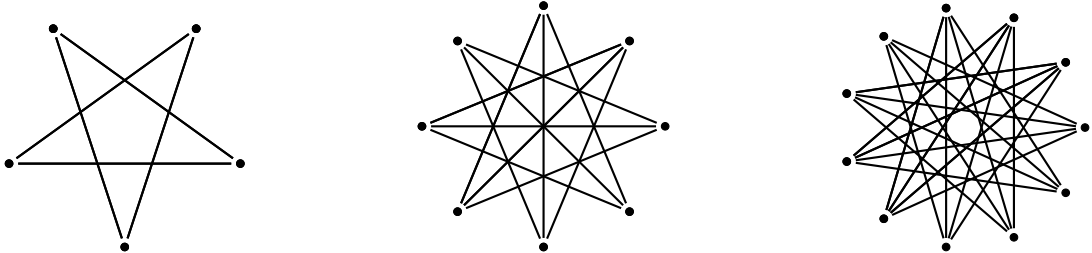
\begin{figure}[ht]
  \centering
    \begin{tikzpicture}[scale = 0.8]
      \foreach \u in {1,...,12}{
        \node [black vertex] (\u) at (360/5 * \u-18:2cm) {};
      }
      \foreach \u in {1,...,5}{
        \foreach \i in {2,...,3}{
          \pgfmathtruncatemacro{\v}{\u + \i};
          \draw [thick] (\u) to (\v);
        }
      }
    \end{tikzpicture}
    \hspace{2cm}
    \begin{tikzpicture}[scale = 0.8]
      \foreach \u in {1,...,15}{
        \node [black vertex] (\u) at (360/8 * \u:2cm) {};
      }
      \foreach \u in {1,...,5}{
        \foreach \i in {3,...,5}{
          \pgfmathtruncatemacro{\v}{\u + \i};
          \draw [thick] (\u) to (\v);
        }
      }
      \end{tikzpicture}
      \hspace{2cm}
      \begin{tikzpicture}[scale = 0.8]
        \foreach \u in {1,...,18}{
          \node [black vertex] (\u) at (360/11 * \u:2cm) {};
        }
        \foreach \u in {1,...,7}{
          \foreach \i in {4,...,7}{
            \pgfmathtruncatemacro{\v}{\u + \i};
            \draw [thick] (\u) to (\v);
          }
        }
      \end{tikzpicture}
    
    \caption{The \andrasfai~graphs $\Gamma_2$, $\Gamma_3$, and $\Gamma_4$.}
    \label{fig:andrasfai_graph}
\end{figure}

\begin{theorem}[Chen--Jin--Koh, 1997]
\label{thm:ChenJinKoh}
  If \(G\) is a triangle-free graph on \(n\) vertices	for which \(\delta(G) > n/3\) and \(\chi(G) \leq 3\),
  then \(G\) is homomorphic to \(\Gamma_d\) for some \(d\).
\end{theorem}

Observe that, together, Theorems~\ref{thm:jin} and~\ref{thm:ChenJinKoh} imply 
that every triangle-free graph on \(n\geq 11\) vertices with minimum degree greater than \(10n/29\) is homomorphic to \(\Gamma_d\) for some~\(d\).

In what follows, given a graph \(H\), we denote by \(\oline{H}\) its 
complement and, for \(S \subseteq V(H)\), we denote by \(H[S]\) the subgraph of \(H\) induced by \(S\). Given \(u \in V(H)\) and \(X \subseteq V(H)\),
we use \(N_X(u)\) to denote the set of neighbors of \(u\) in \(X\).
When \(X = V(H)\), we simply write \(N(u)\), and use \(N[u]\) to denote \(N(u) \cup \{u\}\).
In this paper we use the following property of~\(\Gamma_d\) which says,
in particular, that an \andrasfai\ graph with a maximal independent 
set~\(D\) admits a \(3\)-coloring having~\(D\) as one of the color classes.

\begin{lemma}
\label{lemma:andrasfai_coloring}
  If \(d\in\mathbb{N}\) and \(D_1\) is a maximal independent set of \(\Gamma_d\), then \(\Gamma_d\) admits a \(3\)-coloring \(\{D_1,D_2,D_3\}\)
  such that \(\oline{\Gamma_d}[D_2\cup D_3]\) has no induced \(C_4\).
\end{lemma}

\begin{proof}
  We first observe that the maximal independent sets of \(\Gamma_d\)
  consist precisely of sequences of \(d\) cyclically consecutive vertices of \(\Gamma_d\).
  Indeed,	by the definition of \(E_d\), two vertices \(u\) and \(v\) are adjacent 
  if and only if \(u\) and \(v\) have (circular) distance at least \(d\).
  Thus, we may assume, without loss of generality, that \(D_1 = \{1,\ldots,d\}\).
  Let \(D_2 = \{d+1,\ldots,2d\}\) and \(D_3 = \{2d+1,\ldots, 3d-1\}\).
  As just observed, \(D_2\) and \(D_3\) are independent sets of \(\Gamma_d\).

  Now, we claim that \(\overline{\Gamma_d}[D_2\cup D_3]\) has no induced \(C_4\).
  For this, we prove that \(\Gamma_d[D_2\cup D_3]\) has no induced matching with two edges.
  Suppose, for a contradiction, that \(\Gamma_d[D_2\cup D_3]\) has an induced
  matching \(M\) with two edges.
  Since \(D_2\) and \(D_3\) are independent sets, the edges of \(M\) must join 
  vertices from \(D_2\) to vertices of \(D_3\).
  Say \(M = \{uv,u'v'\}\) with \(u,u' \in D_2\) and \(v,v'\in D_3\).
  By the definition of \(E_d\), either \(N_{D_3}(u) \subseteq N_{D_3}(u')\) 
  or \(N_{D_3}(u') \subseteq N_{D_3}(u)\).
  Assume, without loss of generality, that \(N_{D_3}(u') \subseteq N_{D_3}(u)\).
  Then \(v' \in N_{D_3}(u)\), and hence \(\Gamma_d[\{u,u',v,v'\}]\) is not a matching,
  a contradiction.
\end{proof}

The next step is to generalize the coloring given by Lemma~\ref{lemma:andrasfai_coloring}
to blow-ups of $\Gamma_d$.
Given a graph \(H\), a \emph{blow-up} of \(H\) is any graph obtained 
from \(H\) by replacing each vertex \(u \in V(H)\) by a (possibly empty) 
independent set \(V_u\), and each edge \(uv \in E(H)\) by the complete 
bipartite graph with bipartition \((V_u, V_v)\).
Formally, a graph \(G\) is a blow-up of a graph \(H\) if \(V(G)\) admits
a partition \(\big\{V_u \subseteq V(G) : u \in V(H)\big\}\) 
for which \(E(G) = \bigcup_{uv \in E(H)} E(K_{V_u,V_v})\),
where \(K_{A,B}\) is the complete bipartite graph with bipartition \((A, B)\).
Observe that if a graph \(G\) is a maximal graph homomorphic to \(H\),
then \(G\) must be a blow-up of \(H\).

\begin{lemma}
\label{lemma:andrasfai_blowup_coloring}
  Let \(G\) be a maximal graph homomorphic to \(\Gamma_d\) for some \(d\in \mathbb{N}\).
  If \(I_1\) is a maximal independent set of \(G\), 
  then \(G\) admits a \(3\)-coloring \(\{I_1,I_2,I_3\}\)
  such that \(\overline{G}[I_2\cup I_3]\) has no induced \(C_4\).
\end{lemma}
\begin{proof}
  Let \(h \colon V(G) \to V(\Gamma_d)\) be a homomorphism from \(G\) to \(\Gamma_d\).
  For each \(i\in [3d-1]\), let \(V_i = h^{-1}(i) = \{u\in V(G) : h(u) = i\}\)
  be the set of vertices of \(G\) mapped to \(i\),
  and put \(\mathcal{V} = \big\{V_i : i\in[3d-1]\big\}\).
  Note that, by the maximality of \(G\), every independent set of \(G\) is mapped 
  to an independent set of \(\Gamma_d\).
  Moreover, every maximal independent set of \(G\) is mapped to a maximal 
  independent set of \(\Gamma_d\).
  Thus \(D_1 = h(I_1) = \{h(u) : u \in I_1\}\) is a maximal independent set
  of \(\Gamma_d\).
  Let \(\{D_1,D_2,D_3\}\) be the \(3\)-coloring of \(\Gamma_d\) given by 
  Lemma~\ref{lemma:andrasfai_coloring}, and, for \(i = 2,3\), let
  \(I_i = h^{-1}(D_i) = \{u \in V(G) : h(u) \in D_i\}\).
  Naturally, \(I_i\) is an independent set of \(G\) for \(i = 1,2,3\),
  and, since \(V(\Gamma_d) = D_1\cup D_2\cup D_3\), we have 
  \(V(G) = I_1 \cup I_2 \cup I_3\).
  Therefore, \(\{I_1,I_2,I_3\}\) is a \(3\)-coloring of \(G\).

  Now, we claim that \(\Gcompl[I_2\cup I_3]\) has no induced \(C_4\).
  For this, we prove that \(G[I_2\cup I_3]\) has no induced matching with two edges.
  Suppose, for a contradiction, that \(G[I_2\cup I_3]\) has an induced 
  matching \(M = \{uv,u'v'\}\).
  Since \(I_2\) and \(I_3\) are independent sets, we may assume, without 
  loss of generality,	that \(u,u' \in I_2\) and \(v,v'\in I_3\).
  By its maximality, \(G\) must be a blow-up of \(\Gamma_d\), and hence
  if \(u\) and \(u'\) are in the same element of \(\mathcal{V}\),
  then \(uv', u'v\in E(G)\), a contradiction.
  Therefore, \(u\) and \(u'\) are in different elements of \(\mathcal{V}\).
  Analogously, \(v\) and \(v'\) are in different elements of \(\mathcal{V}\).
  This implies that \(\Gamma_d\big[\{h(u),h(u'), h(v), h(v')\}\big]\) is 
  a matching with two edges, a contradiction.
\end{proof}

\section{Dense graphs with bounded maximum degree}
\label{sec:main-result}

In this section, we prove Theorem~\ref{thm:mainChi}.
The proof is divided into two steps.
First we use that if \(G\) is a triangle-free graph with \(n\) vertices,
independence number \(2\), and maximum degree less than \(19n/29 - 1\),
then \(\Gcompl\) admits a \(3\)-coloring as in Lemma~\ref{lemma:andrasfai_blowup_coloring};
next, we show that every graph \(G\) whose complement admits such a \(3\)-coloring
contains an immersion of~\(K_{\chi(G)}\).
For that, given a positive integer~\(k\),
a \emph{\(k\)-clique coloring} of a graph \(G\) is a partition \(\{D_1,\ldots, D_k\}\) of \(V(G)\)
such that \(D_i\) is a clique of \(G\) for every \(i\in[k]\).

For \(X,Y \subseteq V(G)\) with \(X \cap Y = \emptyset\), 
we denote by \(G[X,Y]\) the bipartite subgraph of \(G\) 
with vertex set \(X \cup Y\) and all edges of \(G\) between \(X\) and \(Y\). 
The proof of the next result uses \(G\) and its complement \(\Gcompl\) at the same time.
To avoid confusion, we write \(\oline{N}_{\!X}(u)\) to refer to 
the vertices in \(X\setminus\{u\}\) that are not adjacent to \(u\) in \(G\),
and \(\oline{N}_{\!X}(Y)\) to refer to the union \(\bigcup_{u\in Y}\oline{N}_{\!X}(u)\),
which is the set of vertices in \(X\) that are nonadjacent in \(G\) to at least one vertex in \(Y\).
Observe that \(\overline{N}\) is precisely the neighborhood function in \(\Gcompl\).

\begin{theorem}
\label{thm:graphs_with_special_clique_coloring}
  Let \(G\) be a graph with independence number~\(2\).
  If \(G\) admits 
  a \(3\)-clique coloring \(\{D_1,D_2,D_3\}\) such that
  (i) \(D_1\) is a maximum clique of \(G\); and
  (ii) \(G[D_2\cup D_3]\) has no induced \(C_4\),
  then \(G\) contains an immersion of a clique whose set of branch vertices 
  is precisely \(D_2\cup D_3\).
\end{theorem}
\begin{proof}
  Let \(E' = E\big(G[D_2\cup D_3]\big)\).
  To find the desired immersion, since \(D_2\) and \(D_3\) are cliques of \(G\), 
  we only need to find a collection of edge-disjoint paths
  \(\{P_{uv} \colon {u\in D_2,\ v\in D_3} \text{ and } {uv\notin E'}\}\),
  where each \(P_{uv}\) is a path in \(G-E'\) joining \(u\) and \(v\).
  In fact, we can find such paths so that each \(P_{uv}\) has length \(3\) and 
  their internal vertices are in \(D_1\).
  For this, for each \(u\in D_2 \cup D_3\), we find a vertex \(r_u \in D_1\) 
  with \(ur_u \notin E(G)\) that ``represents'' \(u\) in \(D_1\),
  meaning that, for every \(uv\notin E(G)\), the path \(P_{uv}\) is 
  the path \(\langle u, r_v, r_u, v\rangle\).
  In what follows, we show how to find these vertices.

  Let \(i\in\{2,3\}\) and let \(C\subseteq D_i\).
  If \({|\oline{N}_{\!D_1}(C)| < |C|}\), then the set
  \({(D_1\setminus \oline{N}_{\!D_1}(C))\cup C}\) is a clique in \(G\)
  larger than \(D_1\), a contradiction to the maximality of \(D_1\).
  So, \({|\oline{N}_{\!D_1}(C)| \geq |C|}\) for every subset \(C\) of \(D_i\).
  Hence, by Hall's Theorem, there is a matching \(M_i\) in
  \(\Gcompl[D_i,D_1]\) that covers \(D_i\).

  Note that for each vertex \(u\in D_2 \cup D_3\) there is precisely one edge 
  in \(M_2\cup M_3\) that contains \(u\), and let \(r_u\in D_1\) be 
  the vertex such that \(ur_u\in M_2\cup M_3\).
  Note that \(r_u \notin N(u)\), and hence, because \(\alpha(G) = 2\), 
  \(r_u\) is adjacent in \(G\) to every non-neighbor of \(u\), that is, to
  every vertex in \(V(G)\setminus N[u]\).

  Let \(u\in D_2\) and \(v\in D_3\).
  Note that if \(r_u = r_v = w\), then \(uv \in E'\), otherwise
  \(u,v,w\) would be an independent set of size \(3\) in \(G\).
  Moreover, if \(uv\notin E'\), then \(r_u\in N(v)\) and \(r_v\in N(u)\),
  and also \(r_ur_v\in E(G)\), because \(r_u \neq r_v\) and \(D_1\) is a clique in \(G\)
  (see Figure~\ref{fig:induced-C4}(a)).

  Now, for every \(u\in D_2\) and \(v\in D_3\) with \(uv\notin E'\), let \(P_{uv}\) 
  be the path \(\langle u, r_v, r_u, v\rangle\) in \(G\).
  We claim that the paths \(P_{uv}\) with \(u\in D_2\), \(v\in D_3\), and
  \(uv\notin E'\) are pairwise edge-disjoint.
  Indeed, let \(u,u' \in D_2\) and \(v,v' \in D_3\) be such that \(uv,u'v' \notin E'\)
  and \(uv \neq u'v'\).
  Note that \(u\) and \(u'\) (resp.\ \(v\) and \(v'\)) are not necessarily distinct,
  but \(u \neq u'\) or \(v \neq v'\).
  If \(v\neq v'\), then \(r_v \neq r_{v'}\) because \(M_3\) is a matching.
  This implies that \(ur_v \neq u'r_{v'}\) (even if \(u = u'\)).
  Analogously, we deduce that if \(u \neq u'\), then \(vr_u \neq v'r_{u'}\).
  In what follows, we prove that \(r_ur_v \neq r_{u'}r_{v'}\).
  Suppose, for a contradiction, that \(r_ur_v = r_{u'}r_{v'}\).
  If \(r_u = r_{u'}\) and \(r_v = r_{v'}\), then \(u=u'\) and \(v=v'\) 
  because \(M_2\) and \(M_3\) are matchings, a contradiction.
  Thus, we must have \(r_u = r_{v'}\) and \(r_v = r_{u'}\).
  As argued in the previous paragraph, this implies that \(uv',vu' \in E'\) 
  (see Figure~\ref{fig:induced-C4}(b)).
  But then \(\{u,u',v,v'\}\) induces a \(C_4\) in \(G[D_2\cup D_3]\),
  a contradiction.

  Since \(D_2\) and \(D_3\) are cliques, and the paths \(P_{uv}\) with \(u\in D_2\),
  \(v\in D_3\), and \(uv\notin E'\) are edge-disjoint, 
  \({G[D_2\cup D_3]\cup\{P_e \colon e\notin E'\}}\) is the desired immersion.
\end{proof}

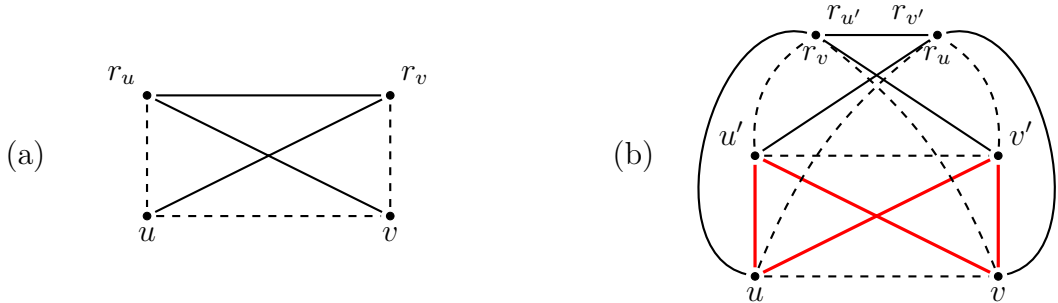
\begin{figure}[ht]
  \centering
    \begin{tikzpicture}[scale = 0.8]
    \begin{scope}[xshift=-10cm,yshift=1cm]
      \node at (-4,0) {(a)};
      \node (u)	[black vertex] at	(-2,-1)	{};
      \node (ru)	[black vertex] at	(-2,1)	{};
      \node (v)	[black vertex] at	(2,-1)	{};
      \node (rv)	[black vertex] at	(2,1)	{};
      \node [anchor=north]		at		(u)		{$u$};
      \node [anchor=south east]	at		(ru)	{$r_u$};
      \node [anchor=north]		at		(v)		{$v$};
      \node [anchor=south west]	at		(rv)	{$r_v$};
      \draw [thick,dashed] (ru) -- (u) -- (v) -- (rv);
      \draw [thick] (v) -- (ru) -- (rv) -- (u);
    \end{scope}
      \node at (-4,1) {(b)};
      \node (u)	[black vertex] at	(-2,-1)	{};
      \node (u')	[black vertex] at	(-2,1)	{};
      \node (v)	[black vertex] at	(2,-1)	{};
      \node (v')	[black vertex] at	(2,1)	{};
      \node (ru)	[black vertex] at	(1,3)	{};
      \node (rv)	[black vertex] at	(-1,3)	{};
      
      \node [anchor=north]		at		(u)		{$u$};
      \node [anchor=south east]	at		(u')	{$u'$};
      \node [anchor=north]		at		(v)		{$v$};
      \node [anchor=south west]	at		(v')	{$v'$};
      \node [anchor=north]		at		(ru)	{$r_u$};
      \node [anchor=north]		at		(rv)	{$r_v$};
      \node [anchor=south east]	at		(ru)	{$r_{v'}$};
      \node [anchor=south west]	at		(rv)	{$r_{u'}$};
  
      \draw [very thick, red] (u) -- (u') (v) -- (v');
      \draw [thick] (ru) to [bend left = 90] (v) (ru) -- (u') (ru) -- (rv);
      \draw [thick] (rv) to [bend right = 90] (u) (rv) -- (v');
      \draw [thick,dashed] (u) -- (v) (u') -- (v');
      \draw [very thick,red] (u) -- (v') (v) -- (u');
      \draw [thick,dashed] (u) to [bend left = 15] (ru) (v) to [bend right = 15] (rv) (u') to [bend left] (rv) (v') to [bend right] (ru);
    \end{tikzpicture}
    \caption{(a) The path \(P_{uv}\) for \(u \in D_2\) and \(v \in D_3\) when \(uv \not\in E'\). Solid (resp.\ dashed) lines illustrate edges (resp.\ missing edges). (b) The induced \(C_4\), in bold red, described in the proof of Theorem~\ref{thm:graphs_with_special_clique_coloring}.
    }
    \label{fig:induced-C4}
\end{figure}

Theorem~\ref{thm:new-main} implies our main result.
For its proof, we need the following theorem due to Gallai~\cite{gallai1963kritische} (see~\cite[Corollary 2]{Stehlik2003}) and, for that, a definition: a graph~$G$ is said 
to be \emph{$k$-critical} if $\chi(G) = k$ 
and $\chi(G-u) < k$, for every $u \in V(G)$.

\begin{theorem}[Gallai, 1963]
\label{thm:critical}
  Every $k$-critical graph whose complement is connected
  has at least \(2k-1\) vertices.
\end{theorem}

\begin{theorem}
\label{thm:new-main}
  Let \(G\) be a graph with independence number at most \(2\). If the complement of \(G\) is homomorphic to \(\Gamma_d\) for some \(d\in\mathbb{N}\), 
  then \(G\) contains a strong immersion of~\(K_{\chi(G)}\).
\end{theorem}
\begin{proof}
  Let \(G\) be a counterexample that minimizes \(|V(G)|\) 
  and let \(k = \chi(G)\).
  In order to apply Theorem~\ref{thm:critical}, we show that $G$ is $k$-critical
  and \(\oline{G}\) is connected.

  Let us first show that $G$ is $k$-critical.
  Indeed, if $G$ is not $k$-critical, then there is a vertex $u \in V(G)$
  such that $\chi(G-u) = k$.
  Thus $G' = G-u$ has independence number at most~2 and its complement 
  is homomorphic to \(\Gamma_d\).
  Hence, \(G'\) contains
  an immersion of~$K_k$ by the minimality of $G$. But this immersion would also be contained in \(G\), a contradiction.
  Therefore, \(G\) is \(k\)-critical.
  Next, we show that \(\Gcompl\) is connected.
  If \(\Gcompl\) is not connected, then $V(G)$ can be partitioned into 
  two non-empty sets $V_1, V_2$ such that $uv \in E(G)$ whenever $u \in V_1$ and $v \in V_2$.
  Let $k_1$ and $k_2$ be the chromatic numbers of $G[V_1]$ and $G[V_2]$, respectively,
  and note that $k = k_1+k_2$.
  Observe that, for \(i \in \{1,2\}\), the graph
  \(G[V_i]\) has independence number at most~2 and its complement is 
  homomorphic to \(\Gamma_d\).
  Thus, by the minimality of \(G\), \(G[V_i]\)~contains an immersion of \(K_{k_i}\).
  But every vertex in $V_1$ is adjacent in $G$ to every vertex in~$V_2$.
  Hence, $G$ must contain an immersion of
  $K_{k_1 + k_2}=K_{k}$, a contradiction.
  Therefore, \(\Gcompl\) is connected.
  As \(G\) is \(k\)-critical and connected, \(|V(G)| \geq 2k-1\) by Theorem~\ref{thm:critical}.

  Let \(H\) be a maximal supergraph of \(\Gcompl\) that is homomorphic to \(\Gamma_d\), and let \(I_1\) be a maximum independent set in \(H\).
  Clearly \(I_1\) is also an independent set in \(\Gcompl\), and hence induces a clique in \(G\). Thus, if \(|I_1| \geq k\), then \(G[I_1]\) is the desired immersion. So assume that \(|I_1| \leq k-1\). By Lemma~\ref{lemma:andrasfai_blowup_coloring}, graph \(H\) admits a \(3\)-coloring \(\{I_1,I_2,I_3\}\) such that~\(\oline{H}[I_2 \cup I_3]\) has no induced~\(C_4\).
  Let \(G'=\oline{H}\) and note that \(G'\) is a subgraph of \(G\). Thus, the independence number of \(G'\) is at least~\(2\). As \(H\) is a maximal graph homomorphic to~\(\Gamma_d\), by the observation just before Lemma~\ref{lemma:andrasfai_blowup_coloring}, \(H\) is a blowup of \(\Gamma_d\), and thus is triangle-free. Hence \(G'=\oline{H}\) has independence number exactly~\(2\). Therefore, by Theorem~\ref{thm:graphs_with_special_clique_coloring}, \(G'\)~contains an immersion of a clique whose set of branch vertices is precisely \(I_2 \cup I_3\). As~\(G'\) is a subgraph of \(G\), this immersion is also contained in \(G\). Now, recall that \(|V(G)| \geq 2k-1\) and \(|I_1| \leq k-1\). Hence, \(|I_2 \cup I_3| = |V(G)|-|I_1| \geq 2k-1 - (k-1) = k\), and \(G\) contains an immersion of \(K_k\), a contradiction.
\end{proof}

Now, we can prove Theorem~\ref{thm:mainCliqueCovering}.

\begin{proof}[Proof of Theorem~\ref{thm:mainCliqueCovering}]
  Observe that \(\Gcompl\) is triangle-free, with \(\delta(\Gcompl) = (n-1) - \Delta(G) > n/3\)
  and~\(\chi(\Gcompl) \leq 3\).
  Hence, \(\Gcompl\) is homomorphic to \(\Gamma_d\) for some \(d\), 
  by Theorem~\ref{thm:ChenJinKoh}.
  The result follows by Theorem~\ref{thm:new-main}.
\end{proof}

Finally, we can prove Theorem~\ref{thm:mainChi}.

\begin{proof}[Proof of Theorem~\ref{thm:mainChi}]
  Note that \(\Gcompl\) is a triangle-free graph with \(\delta(\Gcompl) = (n-1) - \Delta(G) > 10n/29 > n/3\).
  Thus, \(\chi(\Gcompl) \leq 3\) by Theorem~\ref{thm:jin}.
  Therefore \(G\) has independence number at most \(2\) and clique covering number at most \(3\).
  Since \({\Delta(G) < 19n/29 -1 < 2n/3 - 1}\),
  the result follows by Theorem~\ref{thm:mainCliqueCovering}.
\end{proof}

\section{Immersions of $K_{2\lfloor n/5\rfloor}$}\label{sec:2/5}

In this section, we present a somewhat simpler proof of Theorem~\ref{thm:gauthier}.
A natural step taken by Gauthier, Le, and Wollan~\cite{gauthier2019forcing}
is to prove that a minimal graph with independence number~\(2\) contains 
an induced copy of \(C_5\).
We prove it here for completeness.

\begin{lemma}\label{lemma:induced-C5}
  If \(G\) is a minimal graph with independence number \(2\), 
  then $G$ has an induced copy of \(C_5\).
\end{lemma}
\begin{proof}
  First, we claim that \(G\) contains an induced path of length \(2\).
  Indeed, since \(G\) is not a complete graph, there is at least a pair of
  nonadjacent vertices \(u\) and \(v\).
  Let \(P\) be a shortest path joining \(u\) and \(v\).
  Note that \(P\) must be an induced path and, since \(u\) and \(v\) are nonadjacent,
  \(P\) contains an induced path of length \(2\) as desired.

  Let \(\langle v_1,v_2,v_3\rangle\) be an induced path of length \(2\) in \(G\).
  Observe that, by the minimality of~\(G\), for any edge $uv\in E(G)$, we have $\alpha(G - {uv})=3$.
  Hence, there is a vertex $v_4$ that is nonadjacent to both $v_1$ and $v_2$;
  and there is a vertex $v_5$ that is nonadjacent to both $v_2$ and $v_3$.
  Since \(v_1\) is nonadjacent to both \(v_3\) and \(v_4\), we have \(v_3v_4\in E(G)\). 
  Hence \(v_4 \neq v_5\).
  Analogously, \(v_1v_5, v_4v_5\in E(G)\), and hence
  \(\langle v_1,v_2,v_3,v_4,v_5,v_1\rangle\) is an induced copy of \(C_5\)
  in \(G\), as desired.
\end{proof}

Now we can present our proof of Theorem~\ref{thm:gauthier}.

\begin{proof}[Alternative proof of Theorem~\ref{thm:gauthier}]
  The proof follows by induction on \(n + |E(G)|\).
  One can easily check that the statement holds for \(n \leq 9\).
  Since we seek an immersion with $2\lfloor n/5\rfloor$ vertices, 
  we may also assume $n=5t$ for some $t\geq2$. So \(n\geq 10\),
  and now we look for an immersion of \(K_{2t}\) in \(G\).

  If \(\alpha(G-e) \leq 2\) for some edge \(e\in E(G)\),
  then, by the induction hypothesis, \(G-e\) (and consequently \(G\)) 
  contains an immersion of \(K_{2t}\), as desired.
  Therefore, we may assume that \(G\) is minimal with \(\alpha(G) \leq 2\).
  Moreover, the minimality of \(G\) implies \(\alpha(G) = 2\),
  and, in particular, \(G\) is not a complete graph.
  By Lemma~\ref{lemma:induced-C5}, there is an induced copy \(C\) of \(C_5\) in~\(G\).
  Let \(C = \langle v_1,v_2,v_3,v_4,v_5,v_1\rangle\), with \(v_i\in V(G)\) for \(i\in [5]\).

  \begin{claim}
    Every vertex in $V(G)\setminus V(C)$ is adjacent to three consecutive vertices in $C$.
  \end{claim}
  \begin{proof}
    Let \(u \in V(G)\setminus V(C)\).
    Assume, without loss of generality, that $u$ is nonadjacent to $v_1$. 
    Since $v_1v_3,v_1v_4\not\in E(G)$ and \(\alpha(G) = 2\),
    $u$ is adjacent to $v_3$ and $v_4$. 
    Again, since $v_2v_5\not\in E(G)$ and \(\alpha(G) =2\), 
    \(u\) is either adjacent to \(v_2\) or to \(v_5\),
    as desired.
    \renewcommand{\qedsymbol}{$\blacksquare$}%
  \end{proof}

  By the induction hypothesis, \(G - V(C)\) contains an immersion \(K'\) of \(K_{2t - 2}\).
  Let \(I \subseteq V(G)\setminus V(C)\) be the branch vertices of \(K'\).
  Partition $V(G)\setminus(I\cup V(C))$ into five sets $Z_1, \ldots, Z_5$ such that 
  if $u\in Z_i$ then $v_{i-1},v_i,v_{i+1}\in N(u)$, with \(v_0=v_5\) and \(v_6=v_1\).
  Observe that a vertex \(u \notin V(C)\) may fit in more than one such \(Z_i\).
  If this is the case, it is included in only one such \(Z_i\) chosen arbitrarily. 
  Thus \(|Z_1| + \cdots + |Z_5| = |V(G)\setminus(I\cup V(C))| = n - (2t - 2 + 5) = 3(t-1)\), 
  and hence there is an $i$ for which \(|Z_i|\leq\frac{3}{5}(t-1).\)
  Assume, without loss of generality, that
  \begin{equation}\label{eq:Zi}
    |Z_2|\leq\frac{3}{5}(t-1).
  \end{equation}

  Now, for \(i \in\{1,3\}\), let 
  \[X_i = I\setminus N(v_i) \qquad\text{and}\qquad Y^+_i=N(v_i)\setminus(I\cup V(C)).\]
  Since $v_1v_3\not\in E(G)$ and \(\alpha(G) = 2\), it follows that \(X_1\cap X_3 = \emptyset\) and
  \begin{equation}\label{eq:Y_1+UY_3+}
    Y^+_1\cup Y^+_3=V(G)\setminus(I\cup V(C)).
  \end{equation}
    
  Since $\alpha(G)= 2$, the set $\overline{N(u)} = V(G)\setminus (N(u)\cup\{u\})$ induces
  a clique for every $u\in V(G)$.
  Thus, if \(|\overline{N(u)}| \geq 2t\) for some vertex \(u\in V(G)\),
  then \(\overline{N(u)}\) induces the desired immersion.
  Therefore, we may assume that \(|\overline{N(u)}| \leq 2t-1\)
  for every \(u\in V(G)\).
  This implies that \(|N(u)| = n - 1 - |\overline{N(u)}| \geq 3t\) for every \(u\in V(G)\).
  The next claim is an important step in this proof.
    
  \begin{claim}\label{claim:good-subsets}
     There are disjoint sets $Y_1\subseteq Y^+_1\setminus Z_2$ and $Y_3\subseteq Y^+_3\setminus Z_2$ 
     such that \(|Y_1|=|X_1|\) and \(|Y_3|=|X_3|\).
  \end{claim}
  \begin{proof}
    Let \(i\in \{1,3\}\), and note that $|N(v_i)| = |Y^+_i|+|I\setminus X_i|+2$.
    Observe that \(|I\setminus X_i| = |I| - |X_i| = 2(t-1) - |X_i|\), and hence
    \begin{equation}\label{eq:lowerbound:Yi+}
      |Y^+_i|\geq |N(v_i)| - |I\setminus X_i| - 2 
      \geq 3t - 2(t-1) + |X_i| - 2 = 
      t+|X_i|.
    \end{equation}
    
    In particular, by~\eqref{eq:Zi}, we obtain that \(|Y^+_i| \geq |Z_2| + |X_i|\).
    Choose $Y_1\subseteq Y^+_1\setminus Z_2$ with $|Y_1|=|X_1|$, 
    giving priority to vertices not in $Y^+_1\cap Y^+_3$.
    This choice implies that either $Y_1\subseteq Y^+_1\setminus Y^+_3$
    or $Y^+_1\setminus Y^+_3 \subseteq Y_1$.
    If $Y_1\subseteq Y^+_1\setminus Y^+_3$, then by~\eqref{eq:Zi} 
    and~\eqref{eq:lowerbound:Yi+}, since \(t\geq 2\),
    we have $|Y^+_3\setminus Z_2| \geq t + |X_3|-\frac{3}{5}(t-1)\geq |X_3|$, 
    and we can choose \(Y_3\) as desired.
    On the other hand, if $Y^+_1\setminus Y^+_3\subseteq Y_1$, 
    then every element in \(Y_1^+\cup Y_3^+\) that is not in \(Y_1\) must be in~\(Y_3^+\).
    Hence, by~\eqref{eq:Y_1+UY_3+},
    we have \(V(G) \setminus \big(I \cup V(C)\big) = Y_1^+ \cup Y_3^+ = Y_1 \cup (Y_3^+\setminus Y_1)\)
    and, since~\(Y_1 \cap (Y_3^+\setminus Y_1) = \emptyset\), we have that
    \(|Y_3^+ \setminus Y_1| = |V(G) \setminus (I \cup V(C))| - |Y_1| = 3(t-1) - |Y_1|\).
    Note that \(|X_1|+|X_3|\leq |I| = 2(t-1)\) because \(X_1\) and \(X_3\) are disjoint sets in~\(I\). Thus
    \begin{align}
        |Y_3^+\setminus (Y_1\cup Z_2)| 
            & \geq |Y_3^+\setminus Y_1| - |Z_2| \nonumber \\
            & = 3(t-1) - |Y_1| - |Z_2|\nonumber \\
            &  \geq 3(t-1)-|X_1| -\frac{3}{5}(t-1) \nonumber \\ 
            &   >    3(t-1)-2(t-1)+|X_3|-(t-1) 
            \geq |X_3|.\label{eq:lower_bound_Y_3}
    \end{align}
    Therefore, we can choose \(Y_3\) as desired.
    \renewcommand{\qedsymbol}{$\blacksquare$}%
  \end{proof}

  Finally, note that every vertex in $V(G)\setminus\big(I\cup V(C)\cup Z_2\big)$ has
  two neighbors in $\{v_2, v_4, v_5\}$.
  Moreover, every vertex in $X_1\cup X_3$ also has two neighbors in $\{v_2, v_4, v_5\}$.
  Therefore, every pair of vertices $u,w$ with $u\in X_1\cup X_3$ and $w\in Y_1\cup Y_3$ 
  has at least one common neighbor \(v_{uw}\) in \(\{v_2, v_4, v_5\}\).

  Now, let \(X_1 = \{x_1,\ldots,x_{\ell_1}\}\) and \(Y_1 = \{y_1,\ldots,y_{\ell_1}\}\),
  and for each \(i \in \{1,\ldots,\ell_1\}\), consider the path
  \(P_{v_1 x_i} = \langle v_1,y_i,v_{x_iy_i},x_i\rangle\) joining \(v_1\) to \(x_i\).
  Analogously, we define paths~\(P_{v_3 x'}\) joining \(v_3\) to each vertex \(x'\in X_3\).
  It is not hard to check that, since \(X_1 \cap X_3 = Y_1 \cap Y_3 = \emptyset\),
  these paths are edge-disjoint, and we can add \(v_1\) and \(v_3\) to \(K'\),
  while joining~\(v_1\) to \(v_3\) through a path in \(C\).
  This yields an immersion of \(K_{2t}\) in \(G\) whose set of branch vertices is \(I\cup \{v_1,v_3\}\),
  as desired (see Figure~\ref{fig:inductionfor2n/5}).
\end{proof}

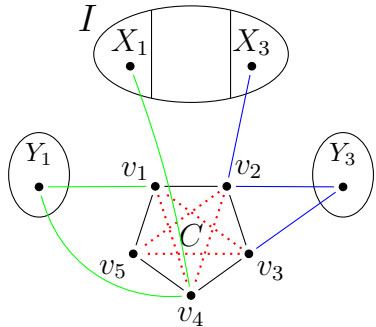
\begin{figure}[H]
  \centering
  \begin{tikzpicture}[scale = 0.8]
\node[black vertex, label={[label distance=-5pt]south east:$v_3$}] (c1) at (-18:1) {};
\node[black vertex, label={[label distance=-5pt]north east:$v_2$}] (c2) at (54:1) {};
\node[black vertex, label={[label distance=-6pt]north west:$v_1$}] (c3) at (126:1) {};
\node[black vertex, label={[label distance=-6pt]south west:$v_5$}] (c4) at (198:1) {};
\node[black vertex, label={[label distance=-3pt]south:$v_4$}] (c5) at (-90:1) {};
\node (C) at (0,0) {$C$};

\draw (c1) -- (c2);
\draw (c1) -- (c3) [thick, red, dotted];
\draw (c1) -- (c4) [thick, red, dotted];
\draw (c1) -- (c5);
\draw (c2) -- (c3);
\draw (c2) -- (c4) [thick, red, dotted];
\draw (c2) -- (c5) [thick, red, dotted];
\draw (c3) -- (c4);
\draw (c3) -- (c5) [thick, red, dotted];
\draw (c4) -- (c5);

\draw[draw] (0, 3) ellipse (1.6 and 0.8);
\draw (-0.65,2.27) -- (-0.65,3.73);
\draw (0.65,2.27) -- (0.65,3.73);
\node at (-1, 3.2) {$X_1$};
\node[black vertex] (x1) at (-1, 2.8) {};
\node at (-1.7, 3.6) [font=\large] {$I$}; 
\node at (1, 3.2) {$X_3$}; 
\node[black vertex] (x3) at (1, 2.8) {};

\draw[draw] (2.5, 1) ellipse (0.5 and 0.7);
\node at (2.5, 1.35) [font=\footnotesize] {$Y_3$};
\node[black vertex] (y1) at (2.5, 0.8) {};

\draw[draw] (-2.5, 1) ellipse (0.5 and 0.7);
\node at (-2.5,1.35) [font=\footnotesize] {$Y_1$}; 
\node[black vertex] (y3) at (-2.5,0.8) {};

\draw (c1) -- (y1) [blue];
\draw (c2) -- (y1) [blue];
\draw (c2) -- (x3) [blue];

\draw (c3) -- (y3) [green];
\draw (c5) to [bend left = 40] (y3) [green];
\draw (c5) to [bend right = 5] (x1) [green];

\end{tikzpicture}
\caption{Paths from $v_1$ to $X_1$ (green), and from $v_3$ to $X_3$ (blue) in the proof of Theorem~\ref{thm:gauthier}. Dotted red lines illustrate missing edges.}
\label{fig:inductionfor2n/5}
\end{figure}

Observe that, while~\eqref{eq:Zi} is used twice in the proof of Claim~\ref{claim:good-subsets},
in both cases we only needed \(|Z_2| \leq t\) to find the desired sets \(Y_1\) and \(Y_3\).
This gap means we could obtain larger sets \(Y_1\) and \(Y_3\), allowing a local improvement as follows.
Suppose we replace the induction hypothesis of an immersion of \(K_{2t}\) by an immersion of \(K_{2t + t/5}\).
The local variables are changed as follows.
The size of \(I\) increases to \(2(t-1) + (t-1)/5\). The minimum degree guarantees that \(|N(v_i)|\) decreases by at most \(t/5\), so \(|N(v_i)| \geq 3t - t/5\). 
Thus~\eqref{eq:lowerbound:Yi+} becomes \(|Y_i^+| \geq 3t/5 + 1/5 + |X_i|\).
In a similar manner, we get \(|X_1| + |X_3| \leq 2(t-1) + (t-1)/5\) and \(|V(G)\setminus (I\cup V(C)))| \geq 3(t-1) - (t-1)/5\),
and then~\eqref{eq:lower_bound_Y_3} becomes \(|Y_3^+ \setminus (Y_1 \cup Z_2)| \geq |X_3|\).
Unfortunately, to conclude the induction step, one should add strictly more than two vertices to the immersion.

\section{Concluding remarks}\label{sec:conclusion}

In this paper, we explore Conjecture~\ref{conj:vergaraChi}
under a maximum degree constraint that allows us to use structural results 
on triangle-free graphs.
These results reveal a connection to the chromatic
number of the complement of these graphs.
A natural possible improvement on our result is 
to weaken the condition on \(\Delta(G)\) in
Theorem~\ref{thm:mainChi} to \(\Delta(G) < 2n/3 - 1\).
This could be approached with an extension of Theorem~\ref{thm:ChenJinKoh},
given by Brandt and Thomassé~\cite{brandt2011dense}, 
that says that triangle-free graphs on~\(n\) vertices with minimum degree greater than \(n/3\) are homomorphic to a Vega graph.
Vega graphs are \(4\)-colorable, and hence have an independent set of size \(n/4\) and their structure can be used to extend Lemmas~\ref{lemma:andrasfai_blowup_coloring} and Theorem~\ref{thm:graphs_with_special_clique_coloring}.

Another possible strategy for improvement is to extend
Theorem~\ref{thm:graphs_with_special_clique_coloring} to graphs with fractional clique number less than \(3\), which would reveal a connection with~\cite{blasiak2007special}.
Indeed, let~\(G\) be a graph for which \(\alpha(G) \leq 2\) and \(\Delta(G) < 2n/3 - 1\).
We claim that \(G\) has fractional clique covering number less than \(3\).
Since \(\Gcompl\) has minimum degree \(\oline{\delta} > n/3\),
the family of neighborhoods \(\{\oline{N}(u) : u \in V(\Gcompl)\}\) 
with constant weight function \(1/\oline{\delta}\) is a fractional coloring of \(\Gcompl\),
and hence \(\Gcompl\) has fractional chromatic number at most \(n/\oline{\delta} < 3\).

\section*{Acknowledgements}
This research has been partially supported by Coordena\c cão de Aperfei\c coamento
de Pessoal de N\'\i vel Superior (CAPES), Brazil, Finance Code 001.
F.~Botler is partially supported by CNPq (304315/2022-2) and CAPES (88881.973147/2024-01).
C.~G.~Fernandes is partially supported by CNPq (310979/2020-0 and 404315/2023-2).
C.~N.~Lintzmayer is partially supported by CNPq (312026/2021-8 and 404315/2023-2) and by L'ORÉAL-UNESCO-ABC For Women In Science.
S. Mishra is supported by Fondecyt Postdoctoral grant $3220618$ of Agencia National de Investigati\'{o}n y Desarrollo (ANID), Chile.
R.~A.~Lopes is supported by CAPES (88887.843699/2023-00).
B.~L.~Netto is supported by CAPES (88887.670803/2022-00).
M.~Sambinelli is partially supported by CNPq (407970/2023-1).
CNPq is the National Council for Scientific and Technological Development of Brazil. 

\bigskip
\noindent
\textbf{E-mail addresses:} \\ \texttt{\{fbotler, cris\}@ime.usp.br} (F.\ Botler, C.\ G.\ Fernandes), \\
\texttt{\{carla.negri, m.sambinelli\}@ufabc.edu.br} (C.\ N.\ Lintzmayer, M.\ Sambinelli),\\
\texttt{\{rui, brunoln\}@cos.ufrj.br} (R.\ A.\ Lopes, B.\ L.\ Netto),\\
\texttt{suchismitamishra6@gmail.com} (S.\ Mishra).

\end{document}